\documentclass[a4j,10pt,onecolumn,oneside,notitlepage,openany,final]{report}
\usepackage{amsmath}
\usepackage{amsthm}
\usepackage{amsfonts}
\usepackage{amssymb}
\usepackage{bm}
\usepackage{graphicx}
\usepackage{longtable}

\newtheorem{thm}[]{Theorem} 

\newtheorem{lem}[thm]{Lemma}

\newtheorem{prop}[thm]{Proposition}
\newtheorem{rem}[thm]{Remark}

\begin{document}
\title{\textbf{On a classification of ideals of local rings for irreducible curve singularities}}
\author{Masahiro WATARI}
\date{}
\maketitle
\begin{abstract}
We consider a classification problem of ideals by codimension in case rings are the local rings of irreducible curve singularities.  
In this paper, we introduce a systematic method to solve this problem. 
\end{abstract}

\section{Introduction}
Let $k$ be an algebraically closed field of characteristic zero. 
In the present paper, we consider a local integral noetherian $k$-algebra $\mathcal{O}$ of dimension one. 
We also assume that $\mathcal{O}$ is complete. 
Our aim  is to present an effective method for a classification of ideals of  $\mathcal{O}$ with fixed codimension.
Such classification is important for the study of punctual Hilbert schemes for irreducible curve singularities. 
See \cite{PS} about punctual Hilbert schemes for curve singularities. 
In \cite{S} and \cite{SW}, the results in this paper were used implicitly. 
However, they were not described at all. 
So we explain the methods precisely in this paper. 
 
In Section\,\ref{Preliminaries}, we fix the notations and prove some lemmas needed later. 
We define the order set of an ideal. 
We also show that the set of all ideals with fixed codimension admits finitely many order sets.  
In Section\,\ref{Determination of ideals}, we consider generators of an ideal whose order set is a given $\Gamma$-module. 
We first consider normal forms of these. 
In general, we need to determine the coefficients in the normal forms to attain the given order set. 
The algorithm to determine the coefficients is also introduced. 
In Section\,\ref{Algorithms}, we first introduce two algorithms. 
One is to compute the minimal generating set for a given $\Gamma$-module 
and the other is to compute all order sets for ideals with  codimension $i$ from those for ideals with  codimension $i-1$. 
These algorithms allow us to construct a systematic way to determine the all order sets for ideals with a given codimension.  
Finally, we prove that Theorem\,\ref{Computational steps} which gives a computational method to determine  all ideals with a given codimention. 
In Section\,\ref{Examples}, we consider two examples, the cases for the singularities of $A_6$ and $E_6$ types.

\section{Preliminaries}\label{Preliminaries}
We first fix notations. 
For a local ring $\mathcal{O}$ as in the previous section, its maximal ideal is expressed by $\mathfrak{m}_{\mathcal{O}}$. 
We denote by $\overline{\mathcal{O}}$ the normalization  of $\mathcal{O}$.   
In our case, we have $\overline{\mathcal{O}}=k[[t]]$. 
We call the set $\Gamma:=\{\mathrm{ord}(f)\,|\,f\in \mathcal{O}\}$ the \emph{semigroup} of $\mathcal{O}$. 
The set $G:=\mathbb{N}\setminus \Gamma$ is called the \emph{gap sequence} of $\Gamma$. 
The \emph{conductor of $\Gamma$} is to be $c:=\max\{G\}+1$. 
A subset $S$ of $\Gamma$ is called a \emph{$\Gamma$-module}, 
if the following two conditions hold: 
(i) $s_1+s_2\in S$ for $\forall s_1, s_2\in S$, (ii) $s+\gamma\in S$ for $\forall s\in S$, $\forall \gamma\in\Gamma$. 
If a $\Gamma$-module $S$ is generated by a set $A=\{\alpha_1,\ldots,\alpha_m\}$ 
(i.e. $S=\{\gamma+\sum_{i=1}^{m} \alpha_i|\,\gamma\in\Gamma\}$), then we express it by $S=\langle A\rangle=\langle \alpha_1,\ldots,\alpha_m\rangle$.  
We refer $A$ as a \emph{generating set} of $S$. 
A generating set $A$ of $S$ is called \emph{minimal}, if $A\setminus \{\alpha\}$ for $\forall \alpha\in A$ does not generate $S$.   
Let $I$ be a nonzero ideal in $\mathcal{O}$. 
We call $\tau:=\tau(I):=\dim_{k}\mathcal{O}/I$ the \emph{codimension} of $I$ and denote by $\mathcal{I}_r$ the set of all ideals with $\tau(I)=r$.   
The set $\Gamma(I):=\{\mathrm{ord}(f)\,|\,f\in I\}$ is called the \emph{order set} of $I$. 
Note that $\Gamma(I)$ is a $\Gamma$-module. 
Put $G(I):=\Gamma\setminus \Gamma(I)$.  
We define the \emph{conductor of $\Gamma (I)$} by $c(I):=\max\{G(I)\}+1$.  
The following lemma shows that the codimension of $I$ depends on $G(I)$.
Namely, it is determined by  $\Gamma(I)$. 
\begin{lem}\label{colength}
An ideal $I$ of $\mathcal{O}$ has $\tau(I)=r$ if and only if we have $\sharp  G(I) =r$. 
\end{lem}
\begin{proof}
An ideal $I$ belongs to $\mathcal{I}_{r}$ if and only if 
\begin{equation}\label{quotient}
\mathcal{O}/I=\left\{a_0+a_1t^{d_1}+\cdots+a_{r-1}t^{d_{r-1}}|\,a_i\in k,d_1<\cdots<d_{r-1}\right\}
\end{equation}
holds.    
The last condition implies $\sharp G(I)=r$. 
\end{proof}

Let $S$ be a $\Gamma$-module.
We denote by $\mathcal{I}(S)$ the set of all ideals of $\mathcal{O}$ with $\Gamma(I)=S$. 
We have the following  proposition. 

\begin{prop}\label{decomposition of Ir}
There exist a finite number of distinct $\Gamma$-modules $S_1,\cdots,S_l$ such that the set $\mathcal{I}_r$ is decomposed as
\begin{equation}\label{decomposition}
\mathcal{I}_r=\bigcup_{i=1}^l \mathcal{I}(S_i)
\end{equation}
where $\mathcal{I}(S_i)\cap \mathcal{I}(S_j)$ for $i\neq j$. 
\end{prop}

\begin{proof}
Fix a codimension $r$. 
Let $I$ be an element of $\mathcal{I}_r$. 
By Lemma\,\ref{colength}, we have $\sharp G(I)=r$. 
Write $G(I)=\{0,d_1,\ldots,d_{r-1}\}$. 
Then it must satisfy the condition (\ref{quotient}).
It is clear that the number of such $d_1,\ldots, d_{r-1}$ is finite.
Hence we have only finite number of $\Gamma$-module of the form $\Gamma \setminus \{0,d_1,\ldots,d_{r-1}\}$.  
\end{proof}

For the decomposition\,(\ref{decomposition}) of $\mathcal{I}_r$, set $\mathfrak{S}_r:=\{S_1,\ldots,S_l\}$. 
Letting $A_i$ be the minimal generating set of $S_i$, 
define $\mathfrak{A}_r:=\{A_1\ldots,A_l\}$. 
Let $I$ be a non-zero ideal  of $\mathcal{O}$ such that $I\neq (1)$.
Set $G_1(I):=G(I)\setminus \{0\}$.
We define $b_1:=\min \{G_1(I)\}$. 
For $b_1$, put $B_1(I):=G_1(I)\cap (b_1+\Gamma)$. 
If $G_1(I)\neq B_1(I)$, then put $b_2:=\min \{G_1(I)\setminus B_1(I)\}$ and  $B_2(I):=\{G_1(I)\setminus (b_1+\Gamma)\}\cap (b_2+\Gamma)$.
We continue this process successively. 
Namely, if $G_1(I)\neq \cup_{i=1}^{j-1}B_i(I)$, 
then we set $b_j:=\min \{G_1(I)\setminus \cup_{i=1}^{j-1}B_i(I)\}$ and  $B_j(I):=\{G_1(I)\setminus \cup_{i=1}^{j-1}B_i(I)\}\cap (b_j+\Gamma)$.
Since $\sharp G(I)<\infty$, there exist a positive integer $n$ such that $G_1(I)=\cup_{i=1}^{n}B_i(I)$. 
For each $i$, we define $d_i:=\max\{B_i(I)\}$.

\section{Determination of ideals}\label{Determination of ideals}
Form this section, we freely use the notations introduced in the previous section. 
Let $\mathcal{O}$ be the local ring of an irreducible curve singularity with a semigroup $\Gamma$. 
For an element $f$ of $\mathcal{O}$, we denote by {\small LC}$(f)$ the leading coefficient of $f$. 
For a given $\Gamma$-module $S$, we determine generators of the ideals in  $I(S)$. 
Let $I$ be an element of $I(S)$. 
For each  $\gamma\in \Gamma (I)=S$, we consider the following element of $I$: 
$$
f_\gamma:=t^{\gamma}+\sum_{j\in G(I),j>\gamma}a_{\gamma,j}t^j.
$$
\begin{lem}\label{normal form}
Let $S$ be a $\Gamma$-module. 
For an ideal $I$ in $\mathcal{I}(S)$, 
we can take 
\begin{equation}\label{generators}
F:=\{f_\gamma \in I|\,\gamma\in S,\ \mathrm{min}\{S\}\le \gamma\le \mathrm{min}\{S\}+c-1\}
\end{equation}
as the set of generators of $I$.
\end{lem}
\begin{proof}
It is clear that $I$ is generated by all $f_\gamma$ with $\gamma\in \Gamma(I)$. 
Note that we can take $f_\gamma=t^{\gamma}$ for $\gamma\ge c(I)$. 
Set $\alpha_1:=\min\{\Gamma(I)\}$. 
It is clear that $c(I)\le \alpha_1 +c$ holds. 
Since there exists $t^{j}$ in $\mathcal{O}$ for $\forall j\ge c$, for  $\gamma\in \Gamma$ with $\gamma\ge \alpha_1+c$,  we can reduce $f_{\gamma}$ by $f_{\alpha _1}$ as
$$
f_\gamma-\sum_{j\ge c}b_jt^jf_{\alpha_1}=0
$$
where $b_j$'s are suitable coefficients. 
Hence we can remove $f_\gamma$ with $\gamma\ge \alpha+c$ from the set of generators of $I$. 
\end{proof}

Let $I$ be an element of $\mathcal{I}(S)$. 
In general,  if  (\ref{generators}) is the set of generators of  $I$, 
then the coefficients of  its elements may satisfy some conditions to attain $\Gamma(I)=S$. 
We denote by $H$ the set of all such conditions. 
We must find $H$ to determine the ideals in $\mathcal{I}(S)$ from the data $S$ 
(see also Remark\,$\ref{remark}$ in Section\,\ref{Examples}). 
To determine $H$ for a given $S$, we introduce the reduction for two elements in $\mathcal{O}$. 
This is an analogue of S-polynomial for given two polynomials (cf. \cite{Cox}).  
For $h_1,h_2\in \mathcal{O}$, consider $V:=\{(\gamma_1,\gamma_2)\in \Gamma\times \Gamma|\, \gamma_1\cdot\mathrm{ord}(h_1)=\gamma_2\cdot\mathrm{ord}(h_2)\}$. 
Let $(\alpha,\beta)$ be the element of $V$ that makes the condition $\gamma_1\cdot\mathrm{ord}(h_1)=\gamma_2\cdot\mathrm{ord}(h_2)$ minimal. 
In the similar manner of the definition of $f_\gamma$, we consider the following elements of $\mathcal{O}$:
$$
g_\gamma:=t^{\gamma}+\sum_{j\in G,j>\gamma}b_{\gamma,j}t^j\quad (\gamma\in \Gamma)
$$
We define the \emph{reduction} of $h_1$ and $h_2$ by
\begin{equation}
\mathrm{Red}(h_1,h_2):=\frac{g_{\alpha}}{\mathrm{{\small LC}}(h_1)}h_1-\frac{g_{\beta}}{\mathrm{{\small LC}}(h_2)}h_2.
\end{equation}

\begin{prop}\label{alg3}
Let $I$ be an ideal in $\mathcal{I}(S)$. 
We rewrite the set $(\ref{generators})$ as $F=\{f_1\ldots,f_l\}$ where 
$$
f_i:=t^{\gamma_i}+\sum_{j\in G(I),j>\gamma}a_{i,j}t^j \quad (i=1,\ldots,l).
$$
The condition set $H$ of $F$  is obtained in a finite number of steps by the following algorithm$:$

\noindent
\rm{\texttt{Input:}} $F=\{f_1,\ldots,f_l\}$\\
\rm{\texttt{Output:}} $H$\\
\rm{\texttt{DEFINE:}} $H:=\emptyset$\\
\rm{\texttt{FOR}} each $i,j$ in $\{1\,\ldots,l\}$ with $i\neq j$ \rm{\texttt{DO}}\\
\quad $R:=$Red$(f_i,f_j)$\\
\qquad\rm{\texttt{WHILE}} $\deg (R)<c(I)$ \rm{\texttt{DO}}\\
\qquad\quad\rm{\texttt{IF}} ord$(R)\notin \Gamma(I)$ \rm{\texttt{THEN}} 
$H:=H\cup\{\mathrm{{\small LC}}(R)=0\}$\\
\qquad\quad\rm{\texttt{ELSE}} 
$R:=\mathrm{Red}\displaystyle{\left(R,\sum_{f_i\in L}b_ig_{\sigma_i}f_i\right)}$ for $L=\{f_i\in F|\,\exists \sigma_i\in\Gamma\textit{ s.t. }\gamma_i+\sigma_i=\mathrm{ord}(R)\}$
\end{prop}
\begin{proof}
For distinct elements $f_i$ and $f_j$ in $F$, we first compute $R_1:=\mathrm{Red}(f_i,f_j)$. 
Note that {\small LC}$(R_1)$ is a polynomial with respect to the coefficients of $f_i$ and $f_j$. 
If $\mathrm{ord}(R_1)\notin S$, then we must have {\small LC}$(R_1)=0$. 
We add it to $H$ and put $R_2:=R_1$. 
On the other hand, if $\mathrm{ord}(R_1)\in S$, 
then we set 
$L_{1}:=\{f_i \in F|\,\exists \sigma_i\in\Gamma\textit{ s.t. }\gamma_i+\sigma_i=\mathrm{ord}(R_1)\}$ 
and reduce $R_1$ by the power series $\sum_{f_i\in L_1}b_ig_{\sigma_i}f_i$ $(b_i\in k)$. 
Put $R_2:=\mathrm{Red}\left(R_1,\sum_{f_i\in L_1}b_ig_{\sigma_i}f_i\right)$. 
Next we apply same arguments to $R_2$. 
If $\mathrm{ord}(R_2)\notin S$, then the leading coefficient {\small LC}$(R_2)$ must be 0 for any value of $b_i$. 
This fact yields some conditions of  the coefficients of $f_i$ and $f_j$. 
We add them to $H$ and rewrite $R_2$ by $R_3$. 
If $\mathrm{ord}(R_2)\in S$, then we set $L_{2}:=\{f_i \in F|\,\exists \sigma_i\in\Gamma\textit{ s.t. }\gamma_i+\sigma_i=\mathrm{ord}(R_2)\}$ and reduce $R_2$ by  $\sum_{f_i\in L_2}b_ig_{\sigma_i}f_i$. 
Put $R_3:=\mathrm{Red}\left(R_2,\sum_{f_i\in L_2}b_ig_{\sigma_i}f_i\right)$. 
In this way, we have $\mathrm{ord}(R_1)< \mathrm{ord}(R_2)<\cdots$. 
If $\mathrm{ord}(R_i)\ge c(I)$, then further reduction yeilds no element of $H$, because any integer $a$ with $a\ge c(I)$ belongs to $\Gamma(I)=S$. 
So these procedures terminate in finite many steps.
Finally, we obtain the set $H$ of conditions. 
\end{proof}

\section{ Computational algorithms}\label{Algorithms}
In this section, we give a computational method to determine $\mathcal{I}_r$. 
\begin{prop}\label{alg1}
For a $\Gamma$-module $S$, we obtain the set $A$ of minimal generators for $S$ in a finite number of steps by the following algorithm$:$

\noindent
\rm{\texttt{Input:}}$S$ \\
\rm{\texttt{Output:}}$A$\\
\rm{\texttt{DEFINE:}}$A:=\emptyset$, $S:=S$\\
\rm{\texttt{WHILE}} $S\neq \emptyset$ \rm{\texttt{DO}}\\
\quad $A:=A\cup\{\min\{S\}\}$\\
\quad $S:=S\setminus \{\min\{S\}+\gamma\,|\,\gamma\in \Gamma\}$
\end{prop}
\begin{proof}
For $\alpha_1:=\min\{S\}$, consider the set $\alpha_1+\Gamma$. 
If $S\neq \alpha_1+\Gamma$, then we put $\alpha_2:=\min\{S\setminus (\alpha_1+\Gamma)\}$. 
We repeat this process as follows: 
If $S\neq \cup_{i=1}^{j-1}(\alpha_i+\Gamma)$, 
then set $\alpha_j:=\min\big\{S\setminus \cup_{i=1}^{j-1}(\alpha_i+\Gamma)\big\}$. 
Then we obtain the descending sequence
$$
S\supsetneq S\setminus (\alpha_1+\Gamma)\supsetneq\cdots \supsetneq S\setminus \cup_{i=1}^j(\alpha_i+\Gamma)\supsetneq\cdots.
$$
Since the infinite set $\{a\in\mathbb{N}|\,a\ge \alpha_1+c\}$ is contained in both $S$ 
and $\alpha_1+\Gamma$, we have $\sharp \{S\setminus (\alpha_1+\Gamma)\}<\infty$. 
Thus there exists a positive integer $n$ such that $S=\cup_{i=1}^n(\alpha_i+\Gamma)$.
This fact guarantee the termination of the algorithm. 
It is clear that the set $A$ obtained by this algorithm  is the minimal generating set of $S$.
\end{proof}

\begin{lem}\label{prp r r+1}
If there exists an ideal $I$ in $\mathcal{I}_{r-1}$ with $\Gamma(I)=\langle \alpha_1,\ldots,\alpha_m\rangle$, 
then so does an ideal $I'$ in $\mathcal{I}_{r}$ with $G(I')=G(I)\cup\{\alpha_i\}$ for each $i$. 
Conversely, if there exists  an ideal $J$ in $\mathcal{I}_{r}$ with $G(J)=\cup_{i=1}^mB_i(J)$, 
then  so does an ideal $J'$ in $\mathcal{I}_{r-1}$ with $G(J')=G(J)\setminus \{d_i\}$ for each $i$. 
\end{lem}
\begin{proof}
Let $I$ be an ideal in $\mathcal{I}_{r-1}$ with $\Gamma(I)=\langle \alpha_1,\ldots,\alpha_m\rangle$. 
For an element $\alpha_i\in \Gamma(I)$, put $S=\{\Gamma(I)\setminus \{\alpha_i\}\}\cup \{\alpha_i+\gamma|\,\gamma\in \Gamma\setminus \{0\}\}$. 
It is easy to check that $S$ is a $\Gamma$-module. 
Consider an ideal $I'$ generated by $f_\gamma$ with 
$\mathrm{min}\{S\}\le \gamma \le \mathrm{min}\{S\}+c-1$. 
We see that $\Gamma(I')=S$ and $G(I')=G(I)\setminus \{\alpha_i\}$. 
Since $I\in \mathcal{I}_{r-1}$, we have $\sharp G(I)=r-1$ by Lemma\,\ref{colength}. 
So $\sharp G(I')=r$. 
We also have $I'\in\mathcal{I}_r$ by Lemma\,\ref{colength}. 
Conversely,  let $J$ be an element of $\mathcal{I}_{r}$ with $G(J)=\cup_{i=1}^mB_i(J)$. 
Note that, for each $i$,   the set $\Gamma (J)\cup \{d_i\}$ is a $\Gamma$-module. 
Indeed, by definition of $d_i$, we see that $d_i+\gamma\in\Gamma(J)$ for any $\gamma\in\Gamma\setminus \{0\}$. 
So the set $\Gamma (J)\cup \{d_i\}$ satisfies the definition of $\Gamma$-module. 
In the same manner as above, we take an ideal $J'$ generated by $f_\gamma$ with 
$\mathrm{min}\{\Gamma (J)\cup \{d_i\}\}\le \gamma \le \mathrm{min}\{\Gamma (J)\cup \{d_i\}\}+c-1$.
It is an element of $\mathcal{I}_{r-1}(\Gamma (J)\cup \{d_i\})$. 
\end{proof}

It follows from Lemma\,\ref{prp r r+1} that the following proposition. 
\begin{prop}\label{alg2}
We obtain the set $\mathfrak{S}_r$ from $\mathfrak{S}_{r-1}$ in a finite number of steps by the following algorithm$:$

\noindent
\rm{\texttt{Input:}}
$\Gamma$ and $\mathfrak{S}_{r-1}=\{S_{1}=\langle A_1\rangle \cdots,S_{l}=\langle A_l\rangle\}$ where $A_{i}=\{\alpha_{i,1},\cdots,\alpha_{i,m(i)}\}$ \\
\rm{\texttt{Output:}} $\mathfrak{S}_r$\\
\rm{\texttt{DEFINE:}} $\mathfrak{S}_{r}:=\emptyset$\\
\rm{\texttt{FOR}} each $i$ in $\{1\,\ldots,l\}$ and each $j$ in $\{1\,\ldots,m(i)\}$ \rm{\texttt{DO}}\\
\quad $S:=[S_i\setminus \{\alpha_{i,j}\}]\cup \{\alpha_{i,j}+\gamma|\,\gamma\in\Gamma\setminus\{0\}\}$\\
\qquad \rm{\texttt{IF}} $S\notin \mathfrak{S}_{r}$ \rm{\texttt{THEN}} $\mathfrak{S}_{r}:=\mathfrak{S}_{r}\cup \{S\}$
\rm{\texttt{ELSE}} do nothing
\end{prop}
\begin{proof}
By the argument in the proof of Lemma\,\ref{prp r r+1}, the set 
$[S_i\setminus \{\alpha_{i,j}\}]\cup \{\alpha_{i,j}+\gamma|\,\gamma\in\Gamma\setminus\{0\}\}$ 
is a order set for some ideal in $\mathcal{I}_r$.  
Since we have $\sharp \mathfrak{S}_r<\infty$ by Proposition\,\ref{decomposition of Ir} and 
a $\Gamma$-module is generated by finely many generators by Proposition\,\ref{alg1}, 
this algorithm terminates and yields $\mathfrak{S}_r$.
\end{proof}
Using Proposition\,\ref{alg3}, \ref{alg1} and \ref{alg2}, 
we obtain the following theorem: 
\begin{thm}[Computational steps for $\mathcal{I}_r$]\label{Computational steps}
For a given codimenson $r$, we obtain $\mathcal{I}_r$ in the following $r$ steps.\\
\rm{\textbf{Step\,1:}} Set $\mathfrak{S}_1=\{\Gamma(\mathfrak{m}_{\mathcal{O}})\}$ and find generators of $\Gamma(\mathfrak{m}_{\mathcal{O}})$ by Proposition\,\ref{alg1}.  \\
\rm{\textbf{Step\,\mathversion{bold}$i$ $(i=2,\ldots,r)$:}} First compute $\mathfrak{S}_{i}$ by Proposition\,\ref{alg2}. Next determine $\mathfrak{A}_i$ by applying Proposition\,\ref{alg1} to each elements in $\mathfrak{S}_{i}$.\\
\textbf{Step\,\mathversion{bold}$r+1$:}  For each $S\in \mathfrak{S}_r$, determine in the coefficients in  the  generators of $I(S)$ by Proposition\,\ref{alg3}.
\end{thm}
\begin{proof}
It is clear that $\mathcal{I}_1=\{\mathfrak{m}_{\mathcal{O}}\}$. 
So we have $\mathfrak{S}_1=\{\Gamma(\mathfrak{m}_{\mathcal{O}})\}$. 
By using Proposition\,\ref{alg1}, we can obtain the minimal generating set $A$ of $\Gamma(\mathfrak{m}_{\mathcal{O}})$.  
Set $\mathfrak{A}_1:=\{A\}$. 
With these datum, we compute $\mathfrak{S}_2$ by using Proposition\,\ref{alg2}. 
Applying Proposition\,\ref{alg1} to each element in $\mathfrak{S}_2$, we get $\mathfrak{A}_2$. 
Continuing this process successively, we obtain $\mathfrak{S}_r$ from $\mathfrak{S}_{r-1}$ for a codimention $r$. 
Finally, for each $S$ in $\mathfrak{S}_r$, consider the generators of $I(S)$ obtained by Lemma\,\ref{generators}. 
Their coefficients are determined by Proposition\,\ref{alg3}. 
\end{proof}

\section{Examples}\label{Examples}
We consider two examples in this section.   
If  $\mathfrak{S}_\tau$ consists of $l$ elements, then we write $S_{\tau,i}$ $(i=1,\ldots,l)$ for these elements. 
We also use the notation $\mathcal{I}_{\tau,i}$ instead of $\mathcal{I}(S_{\tau,i})$. 
We first consider the $A_{2d}$ type singularity  
(i.e. the curve singularity whose local ring  $\mathcal{O}$ is isomorphic to $k[[t^2,t^{2d+1}]]$). 
Similar to Lemma\,\ref{normal form},  we can prove the following lemma for the $A_{2d}$ type singularity:
\begin{lem}\label{normal form of ideal}
Let $I$ be an ideal of $k[[t^2,t^{2d+1}]]$.
If we have $\Gamma(I)=\langle \alpha_1,\alpha_2\rangle$ where $\alpha_1<\alpha_2$,
then $I$ is generated by  
\begin{equation}\label{coe}
f_{\alpha_1}=t^{\alpha_1}+\sum_{j\in G(I),j>\alpha_1}a_jt^j,\quad f_{\alpha_2}=t^{\alpha_2}.
\end{equation}
On the other hand, if $I$ has $\Gamma(I)=\langle \alpha_1\rangle$,
then $I$ is generated by $f_{\alpha_1}$ above.
\end{lem}
\begin{rem}\label{remark}
Lemma\,\ref{normal form of ideal} implies that, for the $A_{2d}$ type singularity, the coefficients in $(\ref{coe})$ are completely determined by $\Gamma(I)$. 
Namely, $H=\emptyset$. 
So we can omit the use of Proposition\,\ref{alg3}  in Step\, $r+1$ in Theorem\,$\ref{Computational steps}$. 
\end{rem}
\noindent
 
Now consider the $A_{6}$ singularity.  
Its local ring $\mathcal{O}$ is isomorphic to $k[[t^2,t^7]]$. 
We classify the elements of $\mathcal{I}_\tau$ for $1\le \tau \le 6$. 
We observe that $\Gamma=\{0,2,4,6,7,8,\ldots\}$ where $c=6$.
Carrying  out Theorem\,\ref{Computational steps}, the order sets  $S_{r,i}$  and the components of $\mathcal{I}_\tau$ are determined as in the following two tables:
\begin{center}
\begin{tabular}{c|l}
$\tau$&order sets\\
1&$S_{1,1}=\langle 2,7\rangle$\\
2&$S_{2,1}=\langle 4,7\rangle$, $S_{2,2}=\langle 2\rangle$\\
3&$S_{3,1}=\langle 6,7\rangle$, $S_{3,2}=\langle 4,9\rangle$\\
4&$S_{4,1}=\langle 7,8\rangle$, $S_{4,2}=\langle 6,9\rangle$, $S_{4,3}=\langle 4\rangle$\\
5&$S_{5,1}=\langle 8,9\rangle$, $S_{5,2}=\langle 7,10\rangle$, $S_{5,3}=\langle 6,11\rangle$\\
6&$S_{6,1}=\langle 9,10\rangle$, $S_{6,2}=\langle 8,11\rangle$, $S_{6,3}=\langle 7,12\rangle$, $S_{6,4}=\langle 6\rangle$\\
\end{tabular} 
\end{center}
\begin{center}
\begin{tabular}{c|l}
$\tau$&components of $\mathcal{I}_{\tau}$\\
1&$\mathcal{I}_{1,1}=(t^2,t^7)$\\
2&$\mathcal{I}_{2,1}=(t^4,t^7)$, $\mathcal{I}_{2,2}=( t^2+at^7)$\\
3&$\mathcal{I}_{3,1}=(t^6,t^7)$, $\mathcal{I}_{3,2}=(t^4+at^7,t^9)$\\
4&$\mathcal{I}_{4,1}=(t^7,t^8)$, $\mathcal{I}_{4,2}=(t^6+at^7,t^9)$, $\mathcal{I}_{4,3}=(t^4+at^7+bt^9)$\\
5&$\mathcal{I}_{5,1}=(t^8,t^9)$, $\mathcal{I}_{5,2}=(t^7+at^8,t^{10})$, $\mathcal{I}_{5,3}=(t^6+at^7+bt^9,t^{11})$\\
6&$\mathcal{I}_{6,1}=(t^9,t^{10})$, $\mathcal{I}_{6,2}=(t^8+at^9,t^{11})$,\\
& $\mathcal{I}_{6,3}=(t^7+at^8+bt^{10},t^{12})$, $\mathcal{I}_{6,4}=(t^6+at^7+bt^9+ct^{11})$\\
\end{tabular} 
\end{center}
In the above table, $a,b$ and $c$ are elements of $k$. 

Next we consider the $E_6$ singularity as second example.
It is the irreducible curve singularity whose local ring $\mathcal{O}$ is $k[[t^3,t^4]]$. 
We have $\Gamma=\{3,4,6,7,8,\ldots\}$ where $c=6$. 
We determine the elements of $\mathcal{I}_\tau$ for $1\le \tau \le 6$.
Applying Theorem\,\ref{Computational steps}, we obtain the following two tables:
\begin{center}
\begin{tabular}{c|l}
$\tau$&order sets\\
1&$S_{1,1}=\langle 3,4\rangle$\\
2&$S_{2,1}=\langle 4,6\rangle$, $S_{2,2}=\langle 3,8\rangle$\\
3&$S_{3,1}=\langle 6,7,8\rangle$, $S_{3,2}=\langle 4,9\rangle$, $S_{3,3}=\langle 3\rangle$\\
4&$S_{4,1}=\langle 7,8,9\rangle$, $S_{4,2}=\langle 6,8\rangle$, $S_{4,3}=\langle 6,7\rangle$, $S_{4,4}=\langle 4\rangle$\\
5&$S_{5,1}=\langle 8,9,10\rangle$, $S_{5,2}=\langle 7,9\rangle$, $S_{5,3}=\langle 7,8\rangle$, $S_{5,4}=\langle 6,11\rangle$\\
6&$S_{6,1}=\langle 9,10,11\rangle$, $S_{6,2}=\langle 8,10\rangle$, $S_{6,3}=\langle 8,9\rangle$, $S_{6,4}=\langle 7,12\rangle$, $S_{6,5}=\langle 6\rangle$\\
\end{tabular} 
\end{center}
\begin{center}
\begin{tabular}{c|l}
$\tau$&components of $\mathcal{I}_\tau$\\
1&$\mathcal{I}_{1,1}=(t^3,t^4)$\\
2&$\mathcal{I}_{2,1}=(t^3,t^6)$, $\mathcal{I}_{2,2}=(t^3+at^4,t^8)$\\
3&$\mathcal{I}_{3,1}=(t^6,t^7,t^8)$, $\mathcal{I}_{3,2}=(t^4+at^6,t^9)$, 
$\mathcal{I}_{3,3}=(t^3+at^4+bt^8,t^6-a^2t^8)$\\
4&$\mathcal{I}_{4,1}=(t^7,t^8,t^9)$, $\mathcal{I}_{4,2}=(t^6+at^7,t^8)$,\\ 
&$\mathcal{I}_{4,3}=(t^6+at^8,t^7+bt^8)$, $\mathcal{I}_{4,4}=(t^4+at^6+bt^9)$\\
5&$\mathcal{I}_{5,1}=(t^8,t^9,t^{10})$, $\mathcal{I}_{5,2}=(t^7+at^8,t^9)$,\\
&$\mathcal{I}_{5,3}=(t^7+at^9,t^8+bt^9)$, $\mathcal{I}_{5,4}=(t^6+at^7+bt^8,t^{11})$\\
6&$\mathcal{I}_{6,1}=(t^9,t^{10},t^{11})$, $\mathcal{I}_{6,2}=(t^8+at^9,t^{10})$, $\mathcal{I}_{6,3}=(t^8+at^{10},t^9+bt^{10})$\\
&$\mathcal{I}_{6,4}=(t^7+at^8+bt^9,t^{12})$, $\mathcal{I}_{6,5}=(t^6+at^7+bt^8+ct^{11},t^{9}+(b-a^2)t^{11})$\\
\end{tabular} 
\end{center}
In the above table, $a,b,c,\in k$.

Here we explain how to determine $\mathcal{I}_{6,5}$ by Step\,7 in Theorem\,\ref{Computational steps}. 
By Proposition\,\ref{alg1}, we see that $S_{6,5}$ generated by 6. 
Let $I$ be an ideal in $ \mathcal{I}_{6,5}$. 
We have  $\Gamma(I)=S_{6,5}=\{6,9,10,12,13,\ldots\}$ where $c(I)=12$ and $G(I)=\{0,1,2,3,4,5,6,7,8,11\}$. 
By Lemma\,\ref{normal form}, 
we can take $F=\{f_\gamma\in I|\, \gamma\in \{6,9,10\}\cup \{\gamma|\,12\le \gamma\le 17\}\}$ as the set (\ref{generators}) of generators of $I$. 
We simply write $f_6=t^6+at^7+bt^8+ct^{11}$, $f_9=t^9+dt^{11}$, $f_{10}=t^{10}+et^{11}$ and $f_{\gamma}=t^\gamma$ for $12\le \gamma\le 17$. 
We first apply Proposition\,\ref{alg3} to $f_6$ and $f_9$. 
Then we have
$$
R_1:=\mathrm{Red}(f_6,f_9)=t^3f_6-f_9=at^{10}+(b-d)t^{11}. 
$$
Note that $10\in\Gamma(I)$. 
Since $10=4+6$ $(4\in\Gamma, 6\in \Gamma(I))$ is the unique expression in $S_{6,5}$, we set $L_1=\{f_6\}$. 
Taking $t^4$ as $g_4$, we  reduce $R_1$ by $pt^4f_6$ $(p\in k)$.  
$$
R_2:=\mathrm{Red}(R_1,pt^4f_6)=\left(\frac{b-d-a^2}{a}\right)t^{11}-bt^{12}-ct^{15}
$$
Since $11\notin \Gamma(I)$, we must have $b-d-a^2=0$. 
Hence, we add  $d=b-a^2$ to $H$.  
Since $\mathrm{ord}(R_2)$ is 12,  further reduction step yields no relations. 
Next we consider the reduction of $f_6$ and $f_{10}$.  
In the same argument as above, we can see that $e=a$. 
Furthermore, it is easy to check that $f_{10}$ is expressed by $f_6$.  
We  also can show that all $f_\gamma$ with $12\le \gamma \le 17$ are expressed by $f_6$. 
Finally, we conclude that $\mathcal{I}_{6,5}=(t^6+at^7+bt^8+ct^{11},t^{9}+(b-a^2)t^{11})$.

\noindent
\small 
Masahiro Watari\\
Division of General Subject\\
Tsuyama National College of Technology\\
624-1 Numa\\
Tsuyamashi Okayama 708-8509  Japan.\\
E-mail: watari@tusyma-ct.ac.jp

\end{document}